\documentclass[11pt,reqno]{amsart} 
\pdfoutput=1

\usepackage{graphicx}

\usepackage[top=1.5in,bottom=1.5in,left=1.5in,right=1.5in]{geometry}

\usepackage{xcolor}

\usepackage{enumitem}

\usepackage{microtype}

\usepackage[labelfont=bf,font=sf]{caption}

\usepackage[font={sf,small},labelfont=md]{subfig}

\usepackage[noadjust,nosort]{cite}

\usepackage{amsmath, amsthm, amssymb}
\usepackage{enumerate}
\usepackage{tikz}
\usepackage{mathtools}
\usepackage{array}
\usepackage{mathrsfs}

\usepackage{polynom}

\usepackage{multicol}

\usepackage{dsfont}

\usepackage{framed}

\usepackage{hyperref}
	\hypersetup{colorlinks=true,linkcolor=blue,citecolor=blue,urlcolor=black}
	
\allowdisplaybreaks

\numberwithin{equation}{section}


\newcommand{\eq}[1]{\begin{align*} #1 \end{align*}}
\newcommand{\eeq}[1]{\begin{align} \begin{split} #1 \end{split} \end{align}}

\newcommand{\R}{\mathbb{R}}

\newcommand{\EE}{\mathbf{E}}

\renewcommand{\c}{\mathcal{C}}

\newcommand{\n}{\mathcal{N}}

\newtheorem{thm}{Theorem}[section]

\newtheorem{cor}[thm]{Corollary}
\newtheorem{lemma}[thm]{Lemma}

\theoremstyle{definition}

\newtheorem{remark}[thm]{Remark}

\def\eps{\varepsilon}

\DeclareMathOperator{\sech}{sech}

\DeclareMathOperator{\e}{e}

\renewcommand{\thefootnote}{\fnsymbol{footnote}}

\title[Symmetry breaking in multi-species SK model]{Replica symmetry breaking in multi-species Sherrington--Kirkpatrick model}



\author{Erik Bates}
\author{Leila Sloman}
\thanks{E.B. was partially supported by NSF grant DGE-114747}
\thanks{L.S. was partially supported by NSF grant DGE-1656518}
\address{\newline Department of Mathematics \newline Stanford University \newline 450 Serra Mall, Bldg 380 \newline Stanford, CA 94305 \newline \textup{\tt ewbates@stanford.edu} \newline \textup{\tt lsloman@stanford.edu}}

\author{Youngtak Sohn}
\address{\newline Department of Statistics \newline Stanford University\newline Sequoia Hall, 390 Serra Mall \newline Stanford, CA 94305\newline \textup{\tt youngtak@stanford.edu}}

\begin{document}
\bibliographystyle{acm}

\renewcommand{\thefootnote}{\arabic{footnote}} \setcounter{footnote}{0}

\begin{abstract}
In the Sherrington--Kirkpatrick (SK) and related mixed $p$-spin models, there is interest in understanding replica symmetry breaking at low temperatures. 
For this reason, the so-called AT line proposed by de Almeida and Thouless as a sufficient (and conjecturally necessary) condition for symmetry breaking, has been a frequent object of study in spin glass theory.
In this paper, we consider the analogous condition for the multi-species SK model, which concerns the eigenvectors of a Hessian matrix. 
The analysis is tractable in the two-species case with positive definite variance structure, for which we derive an explicit AT temperature threshold.
To our knowledge, this is the first non-asymptotic symmetry breaking condition produced for a multi-species spin glass.
As possible evidence that the condition is sharp, we draw further parallel with the classical SK model and show coincidence with a separate temperature inequality guaranteeing uniqueness of the replica symmetric critical point.
\end{abstract}

\maketitle

\section{Introduction}
Spin glass theory, originally developed to study disordered magnets \cite{edwards-anderson75}, now includes applications in biology \cite{parisi90,barra-agliari10,agliari-barra-guerra-moauro11}, computer science \cite{nishimori01,mezard-montanari09}, machine learning \cite{hopfield82,amit-gutfreund-sompolinsky85,bovier-picco98,barra-guerra08,barra-genovese-guerra10,agliari-barra-galluzzi-guerra-moauro12,barra-genovese-guerra-tantari12} and econometrics/quantitative sociology \cite{krapivsky-redner03,contucci-ghirlanda07,contucci-gallo-menconi08,barra-contucci10,barra-agliari12}, primarily due to interest in large-scale networks.
However, its prototypical mathematical model, namely that of Sherrington and Kirkpatrick (SK) \cite{sherrington-kirkpatrick75}, is fully mean-field and thus fails to capture the effect of global inhomogeneities and communities.
In order to examine models more faithful to real-world networks, physicists and mathematicians have in recent years advanced the study of bipartite or more general ``multi-species" spin systems, e.g.~\cite{gallo-contucci08,barra-genovese-guerra11,fedele-contucci11,fedele-unguendoli12,barra-galluzzi-guerra-pizzoferrato-tantari14,barra-contucci-mingione-tantari15}.
An ongoing task is to adapt results from classical spin glasses, such as SK, to their multi-species extensions, especially those regarding the so-called glassy phase observed at low temperatures.


This paper focuses on the multi-species SK (MSK) model, which allows arbitrary interactions between sets of binary spins we call ``species" but remains mean-field in the sense that the number of species is fixed even as the population of each species grows to infinity.
This spin system was introduced by Barra, Contucci, Mingione, and Tantari \cite{barra-contucci-mingione-tantari15}, who also proposed a Parisi formula for the limiting free energy when the interaction parameters satisfy a convexity condition.
This formula was proved by Panchenko \cite{panchenko15I}, which allows us to proceed rigorously in the present work.
By entropic considerations, it is known that there does indeed exist a low temperature phase, i.e.~where the disorder is said to be ``symmetry breaking" \cite[Proposition 4.2]{barra-contucci-mingione-tantari15}.
Our main purpose is to prove a quantitative version of this fact.
We are able to do so in the two-species model, for which we find the analogue of the de Almeida--Thouless (AT) line \cite{almeida-thouless78} from the classical SK model.
This is the content of Theorem \ref{2speciessymmetrybreaking}.
For three or more species, our calculations still predict an AT condition given in Corollary \ref{general_RSB_condition}; however, we have been unable to translate this condition into an explicit temperature threshold outside the two-species case.
Section \ref{3plus_difficulties} outlines the relevant difficulties.

The models under consideration are defined below in Section \ref{definitions}.
Our main results for the two-species SK model, namely Theorems \ref{uniquenessof2speciesrRSsol} and \ref{2speciessymmetrybreaking}, are stated in Section~\ref{main_results}.
Their proofs are given in Sections \ref{sec:uniqueness} and \ref{sec:hessian}, respectively, and related results from the literature on single-species models are discussed in Section \ref{subsec:related_results}.

\subsection{The SK and MSK models} \label{definitions}
We consider a collection of Ising spins $\sigma = (\sigma_1,\dots,\sigma_N) \in \{\pm 1\}^N$, subject to the Hamiltonian
\eeq{ \label{hamiltonian}
H_N(\sigma) = \frac{\beta}{\sqrt{N}}\sum_{i,j=1}^N g_{ij}\sigma_i\sigma_j + h\sum_{i=1}^N \sigma_i,
}
where $\beta > 0$ is the inverse temperature, $h \geq 0$ is the external field, and the disorder parameters $g_{ij}$ are independent, centered Gaussian random variables.
In the SK model, these parameters all have unit variance.
In the MSK model, their variances depend on $i$ and $j$ in the following way.
The spins are partitioned into $M \geq 2$ sets as $\{1,\dots,N\} = \bigcup_{s=1}^M I_s$, and then
\eq{
\EE(g_{ij}^2) = \Delta_{st}^2 \quad \text{whenever $i \in I_s$ and $j \in I_t$}.
}
Thus $\Delta^2 = (\Delta_{st}^2)_{1\leq s,t\leq M}$ is a symmetric $M\times M$ matrix.
Considering the infinite volume limit, we assume that
\eq{
\lim_{N\to\infty} \frac{|I_s|}{N} = \lambda_s \in (0,1) \quad \text{for each $s = 1,\dots,M$.}
}
Of central interest is the free energy of the system,
\eq{
F_N = \frac{\EE \log Z_N}{N}, \quad \text{where} \quad Z_N = \sum_{\sigma \in \{\pm1\}^N} \e^{\beta H_N(\sigma)}.
}
Under the assumption that $\Delta^2$ is nonnegative definite, Barra \textit{et al.}~\cite[Theorem~1.2]{barra-contucci-mingione-tantari15} prove that $F_N$ converges as $N\to\infty$, and Panchenko \cite[Theorem 1]{panchenko15I} verifies their prediction that the limit is given by the variational formula
\eeq{ \label{variational_pre}
\lim_{N\to\infty} F_N = \inf \mathscr{P},
}
where $\mathscr{P}$ generalizes the famous Parisi formula \cite{parisi79,parisi80} proved by Talagrand \cite{talagrand06} for the SK model.
In fact, Panchenko shows $\lim_{N\to\infty} F_N \geq \inf \mathscr{P}$ for general $\Delta^2$.
It is the upper bound $F_N \leq \inf \mathscr{P}$, proved in \cite[Theorem 1.3]{barra-contucci-mingione-tantari15} using Guerra's interpolation method \cite{guerra03}, that requires the nonnegative definiteness assumption.

Let us now define $\mathscr{P}$ precisely.
Given an integer $k \geq 0$, consider a sequence
\begin{subequations} \label{parameters}
\begin{align}
0 = \zeta_0 < \zeta_1 < \cdots < \zeta_{k} < \zeta_{k+1} = 1, \label{parameters_zeta}
\intertext{and for each species $s = 1,\dots,M$, a corresponding sequence}
0 = q_0^s \leq q_1^s \leq \cdots \leq q_{k+1}^s \leq q_{k+2}^s = 1. \label{parameters_q}
\end{align}
\end{subequations}
With these parameters, for each $0 \leq \ell \leq k+2$ we define
\eeq{ \label{Qdef}
Q_{\ell} \coloneqq \sum_{s,t = 1}^M \Delta_{st}^2 \lambda_s \lambda_t q_{\ell}^s q_{\ell}^t, \qquad Q_{\ell}^s \coloneqq 2 \sum_{t =1}^M \Delta_{st}^2 \lambda_t q_{\ell}^t, \quad 1 \leq s \leq M,
}
and then
\eq{
X_{k+2}^s \coloneqq \log \cosh\bigg(h+\beta\sum_{\ell=0}^{k+1} \eta_{\ell+1}\sqrt{Q_{\ell+1}^s-Q_{\ell}^s}\bigg).
}
where $\eta_1,\dots,\eta_{k+2}$ are i.i.d.~standard normal random variables.
Next we inductively define
\eeq{ \label{X_def}
X_\ell^s \coloneqq \frac{1}{\zeta_\ell} \log \EE_{\ell+1} \exp(\zeta_\ell X_{\ell+1}^s), \qquad 0 \leq \ell \leq k+1,
}
where $\EE_{\ell+1}$ denotes expectation with respect to $\eta_{\ell+1}$. 
When $\ell=0$, \eqref{X_def} is understood to mean
\eq{
X_0^s = \lim_{\zeta\searrow0}\frac{1}{\zeta}\log \EE_{1}\exp(\zeta X_1^s) = \EE_1(X_1^s).
}
Finally, we can write
\eeq{ \label{parisi_expression}
\mathscr{P}(\zeta,q) \coloneqq \log{2} + \sum_{s = 1}^M \lambda_s X_0^s - \frac{\beta^2}{2} \sum_{\ell=1}^{k+1} \zeta_{\ell}(Q_{\ell+1} - Q_{\ell}),
}
so that \eqref{variational_pre} reads as
\eeq{ \label{variational}
\lim_{N\to\infty} F_N = \inf_{\zeta,q} \mathscr{P}(\zeta,q).
}

The physical interpretation is as follows. 
Each species has an order parameter $\mu_s$, which is the limiting distribution of the overlap
\eeq{ \label{overlap_def}
R_s(\sigma^1,\sigma^2) = \frac{1}{|I_s|}\Big|\sum_{i\in I_s} \sigma_i^1\sigma_i^2\Big|,
}
where $\sigma^1,\sigma^2$ are independent samples from the Gibbs measure associated to \eqref{hamiltonian}.
If the infimum in \eqref{variational} is achieved at $(\zeta,q)$, then $\mu_s = \sum_{\ell=1}^{k+1} (\zeta_{\ell}-\zeta_{\ell-1})\delta_{q_\ell^s}$. 
When $k=0$ and $\mu_s$ is a single atom for every $s$, we will say the system is in the ``replica symmetric" (RS) phase.
Otherwise, we will say the system has ``replica symmetry breaking" (RSB).
In the SK model (i.e.~$M=1$), there is a single order parameter $\mu$, and so it makes sense to discuss the \textit{level} of symmetry breaking.
That is, if $\mu$ consists of $k+1$ distinct atoms, then the model is said to exhibit ``$k$-step replica symmetry breaking" ($k$RSB); alternatively, if $\mu$ has infinite support---so the infimum in \eqref{variational} is not achieved---then there is ``full replica symmetry breaking" (FRSB).
In the MSK model, it may be the case that if $s\neq t$, then $\mu_s$ and $\mu_t$ can have a different number of atoms in their support.
That is, it is possible that for some $\ell$, one has $q^s_{\ell-1} < q^s_{\ell}$ but $q^t_{\ell-1} = q^t_{\ell}$.
Part of what is shown in \cite{panchenko15I}, however, is that $\zeta_\ell-\zeta_{\ell-1}$ is fixed across species.
It is thus reasonable to say that the MSK model exhibits $k$RSB if \eqref{variational} has a minimizer of the form \eqref{parameters}.

For more on the relationship between replica overlaps and the Parisi minimizer, we refer the reader to \cite{auffinger-chen15I,auffinger-chen15II,auffinger-chen-zeng20}, or to \cite{mezard-parisi-virasoro87,talagrand11I,talagrand11II,panchenko13} for extended treatment of the subject.

\subsection{Statements of main results} \label{main_results}
Consider the replica symmetric expression for the free energy, which involves a single parameter $q^s \in [0,1]$ for each species $s$.
By this we mean that in \eqref{parameters}, we set $k=0$ and $q_1^s = q^s$.
Using the formula for the moment generating function of the Gaussian distribution, one finds that
\eeq{ \label{RS_formula}
\mathscr{P}_\mathrm{RS}(q) = \log 2 + \sum_{s=1}^M \lambda_s\Big[\EE_1\log\cosh(\beta\eta_1\sqrt{Q_1^s}+h) + \frac{\beta^2}{2}(Q_2^s-Q_1^s)\Big] - \frac{\beta^2}{2}(Q_2-Q_1).
}
By differentiating this expression with respect to each $q^t$, and then applying Gaussian integration to write
\eq{
\EE_1[\eta_1\tanh(\beta\eta_1\sqrt{Q_1^s}+h)] 
&= \beta\sqrt{Q_1^s}\EE_1\sech^2(\beta\eta_1\sqrt{Q_1^s}+h) \\
&= \beta\sqrt{Q_1^s}\big(1-\EE_1\tanh^2(\beta\eta_1\sqrt{Q_1^s}+h)\big),
}
it follows that any critical point $q$ must satisfy
\eeq{ \label{P_deriv}
\frac{\partial\mathscr{P}_\mathrm{RS}}{\partial q^t} = \beta^2\lambda_t\sum_{s=1}^M\Delta_{st}^2\lambda_s\big[q^s-\EE_1\tanh^2(\beta\eta_1\sqrt{Q_1^s}+h)\big] = 0,
\quad t = 1,\dots,M.
}
If $\Delta^2$ is invertible, then this system implies
\eeq{ \label{RS_condition}
q^s = \EE_1\tanh^2(\beta\eta_1\sqrt{Q_1^s}+h), \quad s = 1,\dots,M.
}
Therefore, we define the set
\eq{
\c(\beta,h) \coloneqq \{q \in [0,1]^M : \EE \tanh^2(\beta\eta\sqrt{Q_1^s}+h) = q_s \text{ for each $s = 1,\dots,M$}\}.
}
As for the SK model, it is not difficult to show that for small $\beta$, $\c(\beta,0)$ is a singleton; obtaining a sharp estimate requires more care. 
We attempt to do so in the two-species case as part of Theorem \ref{uniquenessof2speciesrRSsol} below.
We also expect $\c$ is a singleton whenever $h > 0$, 
and can prove such a statement when $M=2$.

To simplify notation and standardize temperature scale, we henceforth assume $M=2$ and
\eeq{ \label{delta_assumptions}
\Delta^2 = \begin{pmatrix}
\Delta_{11}^2 & 1 \\
1 & \Delta_{22}^2,
\end{pmatrix}
\quad
\text{where}
\quad
\Delta_{11}^2\Delta_{22}^2 > 1
\quad
\text{and} 
\quad
\lambda_1\Delta_{11}^2 \geq \lambda_2\Delta_{22}^2,
}
in particular ensuring \eqref{variational}.
The second inequality above is made without loss of generality, simply by relabeling the species if necessary.
With these assumptions, we can now state our first result.

\begin{thm} \label{uniquenessof2speciesrRSsol}
Assume \eqref{delta_assumptions}.
If either $h > 0$ or
\eeq{ \label{uniqueness_condition}
\beta^2 < \frac{1}{\lambda_1\Delta_{11}^2+\lambda_2\Delta_{22}^2 + \sqrt{(\lambda_1\Delta_{11}^2-\lambda_2\Delta_{22}^2)^2+4\lambda_1\lambda_2}},
}
then $\c(\beta,h) = \{q_*\}$ is a singleton.
In this case,
\eeq{ \label{RS_minimizing}
\mathrm{RS}(\beta,h) \coloneqq \min_{q\in[0,1]^2}\mathscr{P}_\mathrm{RS}(q) = \mathscr{P}_\mathrm{RS}(q_*).
}
\end{thm}

The proof of Theorem \ref{uniquenessof2speciesrRSsol} is provided in Section \ref{sec:uniqueness}.
Based on analogy with the SK model (see Remark \ref{rmk:analogy} below), one might suspect that \eqref{uniqueness_condition} defines the RS phase of the MSK model.
This suspicion is supported by the striking similarity with \eqref{1RSB_condition}, which at least in the SK model is believed to define the RSB phase (see Section \ref{subsec:related_results}).
Notice that when $h=0$ and $q_s = 0$ for each $s$, the quantity $\gamma_s$ defined in Theorem~\ref{2speciessymmetrybreaking} reduces to $\lambda_s$, and the right-hand sides of \eqref{1RSB_condition} and \eqref{uniqueness_condition} are the same.

\begin{thm}\label{2speciessymmetrybreaking}
Assume \eqref{delta_assumptions} and $h>0$.
Let $q = q_*$ be the critical point from Theorem \ref{uniquenessof2speciesrRSsol}.
Define
\eq{
\gamma_s\coloneqq\lambda_s \EE\, \mathrm{sech}^4\,(\beta\eta\sqrt{Q_1^s}+h), \quad s=1,2.
}
If
\eeq{
\beta^2>\frac{1}{\gamma_1 \Delta_{11}^2 +\gamma_2 \Delta_{22}^2 +\sqrt{(\gamma_1 \Delta_{11}^2 -\gamma_2 \Delta_{22}^2)^2 +4 \gamma_1 \gamma_2}}, \label{1RSB_condition}
}
then
 \eeq{\label{1RSB}
\lim_{N\to\infty} F_N < \mathrm{RS}(\beta,h).
}
\end{thm}

\begin{remark} \label{rmk:analogy}
If $\Delta_{11}^2=\Delta_{22}^2=1$, then \eqref{uniqueness_condition} recovers the analogous result for the SK model, proven independently in \cite{guerra01} and \cite{latala02}.
Similarly, \eqref{1RSB_condition} recovers the AT condition proven in \cite{toninelli02}.
Notational choices in the SK model, however, sometimes replace $\beta^2$ with $\beta^2/2$.
\end{remark}

The proof of Theorem \ref{2speciessymmetrybreaking} is given in Section \ref{sec:hessian}.
In terms of the overlap order parameter \eqref{overlap_def}, Theorem \ref{2speciessymmetrybreaking} says that if \eqref{1RSB_condition} holds, then $R_s(\sigma^1,\sigma^2)$ has a limiting distribution that is nontrivial (i.e.~at least two points in the support) for some species $s$.
Moreover, our analysis in the two-species case suggests that \eqref{1RSB} is realized at smaller $\beta$ when both species break symmetry as opposed to just one; see \eqref{2speciescondition_new}, which is equivalent to \eqref{1RSB_condition}.
Therefore, it is plausible that under convexity \eqref{delta_assumptions}, symmetry breaking of one species implies symmetry breaking of all species.
Verifying this statement rigorously, however, requires a global analysis that goes beyond the perturbative approach of this paper.

In the context of the classical SK model, the inequality \eqref{1RSB_condition} is obtained in \cite{toninelli02} as the necessary and sufficient condition for a certain second derivative to be positive.
The analogous object in the MSK model is an $M\times M$ Hessian matrix, and the relevant condition is, at least intuitively, the positivity of its top eigenvalue.
A difficulty posed by multi-dimensionality, however, is that the associated eigenvector need not have all positive entries, which is ultimately needed to conclude symmetry breaking because of the ordering in \eqref{parameters_q}.
When $M=2$, we are able to overcome this difficulty and prove the relevant eigenvector \textit{does} have positive coordinates, by direct analysis of matrix entries.

\subsection{Related results for Ising spin glasses} \label{subsec:related_results}

In the single-species case, \eqref{hamiltonian} is often replaced by a more general Hamiltonian that considers interactions not just between pairs of spins, but also between $p$-tuples of spins for any $p \geq 2$.
More precisely, the mixed $p$-spin model with mixture $\xi(t) = \sum_{p\geq2} \beta_p^2t^p$, inverse temperature $\beta > 0$, and external field $h \geq 0$, has the Hamiltonian
\eq{
H_N(\sigma) = \beta\sum_{p=2}^\infty \frac{\beta_p}{N^{(p-1)/2}}\sum_{i_1,\dots,i_p=1}^N g_{i_1\cdots i_p}\sigma_{i_1}\cdots\sigma_{i_p} + h\sum_{i=1}^N\sigma_i, \quad \sigma \in \{\pm 1\}^N,
}
where the disorder variables $g_{i_1\cdots i_p}$ all i.i.d.~standard normals.
For such models (with suitable decay conditions on the $\beta_p$), the Parisi formula has been proved by Talagrand \cite{talagrand06} when $\beta_p = 0$ for all odd $p$, and by Panchenko \cite{panchenko14} in the general case.

In this more general setting, the set of RS critical points is
\eq{
\c(\beta,h) \coloneqq \{q \in [0,1] : \EE \tanh^2(\beta \eta \sqrt{\xi'(q)} + h) = q\},
}
As discussed in \cite{jagannath-tobasco17I}, the size of $\c(\beta,h)$ is very difficult to determine in general.
Nevertheless, we are more generally concerned with the quantity
\eq{
\alpha(\beta,h) \coloneqq \min_{q \in \c(\beta,h)} \beta^2 \xi''(q)\EE \sech^4(\beta \eta \sqrt{\xi'(q)} + h).
}
The result of Toninelli \cite{toninelli02} for the SK model, which Theorem \ref{2speciessymmetrybreaking} generalizes,  can then be written as
\eq{
\alpha(\beta,h) > 1 \quad \implies \quad \lim_{N\to\infty} F_N< \inf_{q \in \c(\beta,h)} \mathscr{P}_\mathrm{RS}(q),
}
where $\mathscr{P}_\mathrm{RS}$ is the Parisi functional restricted to Dirac delta measures;
see \cite{jagannath-tobasco17I} by Jagannath and Tobasco, who extend this result to mixed $p$-spin models.
It is conjectured that the converse is also true, at least when $\beta_2 > 0$, and some partial results are given in \cite{aizenman-lebowitz-ruelle87,guerra-toninelli02II,talagrand02,jagannath-tobasco17I}.
However, for the Ghatak--Sherrington model (in which spin $0$ is allowed), the converse is known to be false by work of Panchenko \cite{panchenko05}.

%
%
%


\subsection{Challenges with three or more species} \label{3plus_difficulties}
As discussed in Section \ref{main_results}, the relationship between \eqref{uniqueness_condition} and \eqref{1RSB_condition} generalizes the one between the analogous thresholds in the classical SK model.
It would be interesting to have a similar result that applies to the $M$-species model for any $M$.
Unfortunately, it is not clear how the techniques used in Sections \ref{sec:uniqueness} and \ref{sec:hessian} could be adapted to handle the case $M\geq3$.

The challenges are mainly linear algebraic.
For instance, if $A$ is an $M\times M$ positive definite matrix, then the signs of the off-diagonal entries of $A^{-1}$ cannot be determined from $\det(A)$, unless $M\leq2$.
This information is crucial in the proof of Theorem~\ref{uniquenessof2speciesrRSsol} (see \eqref{inverse_A} and \eqref{B_inverse}), which establishes the uniqueness of the RS critical point.
Without this uniqueness, Theorem \ref{2speciessymmetrybreaking} can only be stated as Corollary \ref{general_RSB_condition}, which is useful in determining symmetry breaking only if one can identify a critical point with minimal energy.  This task is non-trivial even when $M=1$; for $M\geq 2$, it is not actually clear \textit{a priori} why there should even be only finitely many critical points.

Another difficulty is analyzing the matrix $K = K(\beta)$ appearing in the proof of Theorem \ref{2speciessymmetrybreaking}.
Determining the exact set of $\beta$ for which $K$ has a positive eigenvalue, as well as an eigenvector with all nonnegative entries, requires one to consider all the relationships among the $M(M+1)/2$ distinct entries of $K$ that yield this property.
Any argument that works for general $M$ will not be able to proceed directly as we do, and may have to first address the following question: If $K$ has a positive eigenvalue, is there necessarily an associated eigenvector with all nonnegative entries?

\subsection{Non-convex cases}
The validity of the Parisi variational formula \eqref{variational} is known only when $\Delta^2$ is nonnegative definite.
This formula fails, for instance, in the bipartite model $\Delta^2 = \begin{pmatrix} 0 & 1 \\ 1 & 0 \end{pmatrix}$ with $h=0$.
There are two ways to see this, both starting from the fact \cite[Theorem 3]{barra-genovese-guerra11} that for this model, the following implication holds:
\eq{
\beta^2\leq\frac{1}{\sqrt{4\lambda_1\lambda_2}} \quad \implies \quad \lim_{N\to\infty} F_N = \log 2 + \beta^2\lambda_1\lambda_2 = \mathscr{P}_\mathrm{RS}(0,0).
}
The first approach is to observe as in \cite[Remark 5]{barra-genovese-guerra11} that $(0,0)$ is a saddle point of $\mathscr{P}_\mathrm{RS}$ rather than a minimizer.
Indeed, in the bipartite model, \eqref{P_deriv} becomes
\eq{
\begin{pmatrix} 
\frac{\partial \mathscr{P}_\mathrm{RS}(q^1,q^2)}{\partial q^1}\vspace{3pt} \\ \frac{\partial \mathscr{P}_\mathrm{RS}(q^1,q^2)}{\partial q^2}
\end{pmatrix}
= \beta^2\lambda_1\lambda_2
\begin{pmatrix}
0 & 1 \\
1 & 0
\end{pmatrix}
\begin{pmatrix}
q^1 - \EE\tanh^2(\beta\eta\sqrt{2\lambda_2q^2}) \\ 
q^2 - \EE\tanh^2(\beta\eta\sqrt{2\lambda_1q^1})
\end{pmatrix}.
}
If $q^1 = 0$ and $q^2>0$, then the partial derivative with respect to $q^2$ is negative, which implies $\mathscr{P}_\mathrm{RS}(0,0) > \mathscr{P}_\mathrm{RS}(0,q^2)$.
Therefore, the right-hand side of \eqref{variational} is strictly smaller than the left-hand side, at least when $\beta^2\leq\frac{1}{\sqrt{4\lambda_1\lambda_2}}$.

The second approach is to observe that if $\Delta^2_{ss} = 0$ for some $s$, then existence of the vector $x\in\R^M$ required in Corollary \ref{general_RSB_condition} holds trivially at all temperatures.
Therefore, in the bipartite model, the proof of Corollary \ref{general_RSB_condition} shows
\eq{
\inf_{\zeta,q} \mathscr{P}(\zeta,q) < \mathscr{P}_\mathrm{RS}(0,0). 
}
Once again, the conclusion is that \eqref{variational} is not correct when $\beta^2\leq\frac{1}{\sqrt{4\lambda_1\lambda_2}}$.

\section{Uniqueness of critical point} \label{sec:uniqueness}

In this section we prove Theorem \ref{uniquenessof2speciesrRSsol}, which asserts the uniqueness of the replica symmetric critical point in the parameter regime $\{h>0\} \cup \{\beta < \beta_0\}$, where $\beta_0$ is identified from \eqref{uniqueness_condition}.
For the SK model, the argument to identify $\beta_0$ is straightforward and can be found in \cite[Section 1.3]{talagrand11I}. 
The two-dimensional nature of the problem here is handled by introducing a pair of inequalities and then optimizing over $\beta$.

Addressing the case $h>0$ is also straightforward in the SK model, at least once the clever Lemma \ref{LatalaGuerralemma} is realized.
In fact, we are able to make use of this lemma once more in the two-species case, since the signs of the entries in a $2\times2$ inverse matrix are easy to determine.

\begin{lemma}[{\cite{guerra01} and \cite{latala02}, see also \cite[Appendix A.14]{talagrand11II}}]\label{LatalaGuerralemma}
Let $\phi$ be an odd, twice-differentiable, increasing, and bounded function, that is strictly concave when $y>0$. 
Then the function $\Phi(x)\coloneqq\EE(\phi(z\sqrt{x}+h)^{2})/x$ is strictly decreasing on $\mathbb{R}^{+}$ and vanishes as $x \rightarrow \infty$.
\end{lemma}

\begin{proof}[Proof of Theorem \ref{uniquenessof2speciesrRSsol}]
Since $\mathscr{P}_\mathrm{RS}(q)$ is a continuous function of $q\in[0,1]^2$, it must attain a minimum.
Since $\mathscr{P}_\mathrm{RS}$ is differentiable on $(0,1)^2$, if this minimum is achieved at some $q \in (0,1)^2$, then $q$ must belong to $\c(\beta,h)$.
Therefore, we will first show that if $h>0$, then \eqref{RS_condition} has at most one solution, and that $\mathscr{P}_\mathrm{RS}$ achieves its minimum in $(0,1)^2$.
In order to handle the case $h=0$, we will separately show that if \eqref{uniqueness_condition} holds, then $\c(\beta,h)$ contains at most one point.
In particular, when $h=0$, it is clear that the single element must be $q_*=(0,0)$, in which case the minimum must be achieved at a boundary point, and we will show that $(0,0)$ is the only possibility.
  
First assume $h>0$.
For ease of notation, we will write $Q^s = Q^s_1$ for $s = 1,2$, where $Q^s_1$ is defined in \eqref{Qdef}.
Considering $Q = \begin{pmatrix} Q^1 \\ Q^2 \end{pmatrix}$ and $q = \begin{pmatrix} q^1 \\ q^2\end{pmatrix}$ as vectors, we have
\eeq{\label{matrixrepofQandq}
Q=Aq,\quad\text{where}\quad A=\begin{pmatrix} 
2\lambda_1\Delta_{11}^2 & 2\lambda_2 \\
2\lambda_1 & 2\lambda_2\Delta_{22}^2
\end{pmatrix}.
}
Since $\Delta^2$ is positive definite, we have $\det A > 0$ and hence
\eeq{\label{inverse_A}
A^{-1}=\begin{pmatrix}
\phantom{-}a&-b\\-c&\phantom{-}d
\end{pmatrix},\quad\text{where}\quad a,b,c,d >0.
}
Assuming $q$ satisfies \eqref{RS_condition}, inversion of $A$ in \eqref{matrixrepofQandq} gives
\begin{subequations} \label{2speciesRSeq}
\begin{align}
aQ^1-bQ^2&=\EE\tanh^{2}(\beta z \sqrt{Q^{1}}+h), \label{2speciesRSeq_1} \\
-cQ^1+dQ^2&=\EE\tanh^{2}(\beta z \sqrt{Q^{2}}+h). \label{2speciesRSeq_2}
\end{align}
\end{subequations}
To show that $\c(\beta,h)$ is a singleton, it suffices (by invertibility of $A$) to show that the system \eqref{2speciesRSeq} admits at most one solution $Q$.

Since $h > 0$, it is clear from \eqref{RS_condition} that $q^1,q^2 > 0$ and thus $Q^1,Q^2 > 0$.
Therefore, we can rewrite \eqref{2speciesRSeq} as
\begin{subequations}
\begin{align}
a-b\frac{Q^2}{Q^1}&=\frac{\EE\tanh^{2}(\beta z \sqrt{Q^{1}}+h)}{Q^1}, \label{rewrite_1}\\
-c\frac{Q^1}{Q^2}+d&=\frac{\EE\tanh^{2}(\beta z \sqrt{Q^{2}}+h)}{Q^2}. \label{rewrite_2}
\end{align}
\end{subequations}
Since $b>0$, the left-hand side of \eqref{rewrite_1} is strictly increasing in $Q^1$ (and approaches $a>0$ as $Q^1\to\infty$), whereas the right-hand side is strictly decreasing by Lemma~\ref{LatalaGuerralemma} (and approaches $0$ as $Q^1\to\infty$).
Therefore, for each fixed value of $Q_2$, there is exactly one value of $Q_1$ satisfying \eqref{rewrite_1}.
Since $c>0$, 
we can also define $Q^1$ from $Q^2$ using \eqref{rewrite_2} instead of \eqref{rewrite_1}.
That is, given $x>0$, define $Q^1(x)$ by
\eq{
-c\frac{Q^1(x)}{x}+d&=\frac{\EE\tanh^{2}(\beta z \sqrt{x}+h)}{x}.
}
Since the right-hand side above is strictly decreasing in $x$, it follows that $Q^1(x)/x$ is strictly increasing in $x$.
In particular, $Q^1(x)$ is strictly increasing in $x$.
Therefore, if we replace \eqref{rewrite_1} by
\eeq{
a-b\frac{x}{Q^1(x)}&=\frac{\EE\tanh^{2}(\beta z \sqrt{Q^{1}(x)}+h)}{Q^1(x)} \label{rerewrite_1},
}
then the right-hand side is strictly decreasing in $x$ by Lemma \ref{LatalaGuerralemma}, while the left-hand side is strictly increasing in $x$.
Consequently, there is at most one value of $x$ such that \eqref{rerewrite_1} holds.

We next check that the minimum of the RS expression $\mathscr{P}_\mathrm{RS}$ is not obtained on the boundary unless $h=0$, in which case our argument will show that $(0,0)$ is the only possible minimizer on the boundary.
Recall from \eqref{P_deriv} that
\eq{
\begin{pmatrix}
y^1 \\ y^2
\end{pmatrix}
\coloneqq
\begin{pmatrix} 
\frac{\partial \mathscr{P}_\mathrm{RS}(q)}{\partial q^1}\vspace{3pt} \\ \frac{\partial \mathscr{P}_\mathrm{RS}(q)}{\partial q^2}
\end{pmatrix}
= B\begin{pmatrix}
x^1 \\ x^2
\end{pmatrix},
}
where
\eq{
B = \beta^2\begin{pmatrix}
\lambda_1^2\Delta_{11}^2 & \lambda_1\lambda_2 \\
\lambda_1\lambda_2 & \lambda_2^2\Delta_{22}^2
\end{pmatrix},
\qquad
x^s = q^s - \EE\tanh^2(\beta\eta\sqrt{Q^s}+h).
}
Because of \eqref{delta_assumptions}, it follows that
\eeq{ \label{B_inverse}
\begin{pmatrix}
x^1 \\ x^2
\end{pmatrix}
=
\begin{pmatrix}
\phantom{-}e & -f \\
-f & \phantom{-}g
\end{pmatrix}
\begin{pmatrix}
y^1 \\ y^2
\end{pmatrix},
\quad \text{where} \quad
e,f,g>0.
}
Suppose toward a contradiction that $\mathscr{P}_\mathrm{RS}$ is minimized as $\mathscr{P}_\mathrm{RS}(q^1,0)$ for some $q^1 \in (0,1]$.
We then have $y^1 \leq 0$ and $x^2 < 0$.
It follows that $y^2$ is negative, since $y^2 \geq 0$ would imply
\eq{
0 > x^2 = -fy^1 + gy^2 \geq gy^2 \geq 0.
}
So now $y^2 < 0$, meaning $\mathscr{P}_\mathrm{RS}(q^1,q^2) < \mathscr{P}_\mathrm{RS}(q^1,0)$ for small enough $q^2 > 0$, giving the desired contradiction.
By similar reasoning, $\mathscr{P}_\mathrm{RS}$ cannot be minimized along $\{0\} \times (0,1]$, $[0,1) \times \{1\}$, or $\{1\} \times [0,1)$.
Moreover, the trivial bound $\tanh^2(u)< 1$ implies that $y^1$ and $y^2$ are both positive if $q^1=q^2=1$, eliminating the possibility that $\mathscr{P}_\mathrm{RS}$ is minimized at $(1,1)$.
Finally, if $h>0$, then the other trivial bound $\tanh^2(h) > 0$ implies that  $y^1$ and $y^2$ are both negative if $q^1=q^2=0$, eliminating the possibility that $\mathscr{P}_\mathrm{RS}$ is minimized at $(0,0)$.
This completes the proof in the case $h > 0$.

For the final part of the proof, we assume \eqref{uniqueness_condition} holds (with $h$ possibly equal to $0$).
Let $F_1(q^1,q^2) \coloneqq \psi(Q^1)$ and $F_2(q^1,q^2) \coloneqq \psi(Q^2)$, where
\eq{
\psi(x) = \EE f(\beta\eta\sqrt{x}+h), \quad \eta \sim \n(0,1), \quad \text{and} \quad f(y) = \tanh^2(y).
}
Using Gaussian integration by parts, we find
\eq{
\psi'(x) = \frac{\beta}{2\sqrt{x}} \EE[\eta f'(\beta\eta\sqrt{x}+h)]
= \frac{\beta^2}{2}\EE f''(\beta\eta\sqrt{x}+h).
}
Also observe that
\eq{
f'(y) = 2\frac{\tanh y}{\cosh^2 y}, \qquad
f''(y) = 2\frac{1-2\sinh^2 y}{\cosh^4 y}.
}
One can check that $f''(y) \in \big[-\frac{2}{3},2\big]$ for all $y\in\R$.
Consequently,
\eeq{ \label{derivative_bounds}
\Big|\frac{\partial F_1}{\partial q^1}\Big| &= 2\lambda_1\Delta_{11}^2|\psi'(Q^1)| \leq 2\beta^2\lambda_1\Delta_{11}^2 
\quad
\Big|\frac{\partial F_1}{\partial q^2}\Big| = 2\lambda_2|\psi'(Q^1)| \leq 2\beta^2\lambda_2 \\
\Big|\frac{\partial F_2}{\partial q^1}\Big| &= 2\lambda_1|\psi'(Q^2)| \leq 2\beta^2\lambda_1 
\quad
\Big|\frac{\partial F_2}{\partial q^2}\Big| = 2\lambda_2\Delta_{22}^2|\psi'(Q^2)| \leq 2\beta^2\lambda_2\Delta_{22}^2.
}
Suppose, toward a contradiction, that $(q^1,q^2)$ and $(p^1,p^2)$ are distinct elements of $\c(\beta,h)$.
That is, each is a fixed point of $(F_1,F_2) : [0,1]^1\to[0,1]^2$.
Let $\gamma(t) = ((1-t)q^1+tp^1,(1-t)q^2+tp^2)$, $0\leq t\leq1$, be the line segment connecting these two points.
We must have
\eeq{ \label{path_integral}
\int_{0}^1 \nabla F_s(\gamma(t)) \cdot (p^1-q^1,p^2-q^2)\ dt = p^s-q^s,\quad s=1,2.
}
On the other hand, the bounds in \eqref{derivative_bounds} reveal that
\eq{
&\int_0^1 |\nabla F_1(\gamma(t)) \cdot (p^1-q^1,p^2-q^2)|\ dt \leq 2\beta^2(\lambda_1\Delta_{11}^2|p^1-q^1| + \lambda_2|p^2-q^2|),
}
as well as
\eq{
&\int_0^1 |\nabla F_2(\gamma(t)) \cdot (p^1-q^1,p^2-q^2)|\ dt \leq 2\beta^2(\lambda_1|p^1-q^1| + \lambda_2\Delta_{22}^2|p^2-q^2|).
}
To derive a contradiction to \eqref{path_integral}, it suffices to show that either
\eq{
2\beta^2(\lambda_1\Delta_{11}^2|p^1-q^1| + \lambda_2|p^2-q^2|) < |p^1-q^1|
}
or
\eq{
2\beta^2(\lambda_1|p^1-q^1| + \lambda_2\Delta_{22}^2|p^2-q^2|) < |p^2-q^2|.
}
By scaling, this is equivalent to showing that for any $t\in[0,\infty]$, either
\eeq{ \label{first_curve}
L_1(t) \coloneqq 2\beta^2\Big(\lambda_1\Delta_{11}^2 + \lambda_2t\Big) < 1
}
or
\eq{
L_2(t) \coloneqq 2\beta^2\Big(\lambda_1\frac{1}{t} + \lambda_2\Delta_{22}^2\Big) < 1.
}
Since $L_1'(t) > 0 > L_2'(t)$ with $L_1(t) \to \infty$ as $t\to\infty$ and $L_2(t) \to \infty$ as $t\to0$,
the maximum value of $\min(L_1(t),L_2(t))$ will be achieved at the unique $t > 0$ such that $L_1(t)=L_2(t)$.
For this value of $t$ we have
\eq{
\lambda_1\Delta_{11}^2 + \lambda_2 t &= \lambda_1\frac{1}{t}+\lambda_2\Delta_{22}^2 \\
\implies \quad t &= \frac{-\lambda_1\Delta_{11}^2+\lambda_2\Delta_{22}^2 + \sqrt{(\lambda_1\Delta_{11}^2-\lambda_2\Delta_{22}^2)^2+4\lambda_1\lambda_2}}{2\lambda_2}.
}
Using this value of $t$ in \eqref{first_curve}, we conclude that a contradiction is realized as soon as
\eq{
2\beta^2\frac{\lambda_1\Delta_{11}^2+\lambda_2\Delta_{22}^2 + \sqrt{(\lambda_1\Delta_{11}^2-\lambda_2\Delta_{22}^2)^2+4\lambda_1\lambda_2}}{2} < 1,
}
which is exactly \eqref{uniqueness_condition}.
\end{proof}

\section{Hessian condition for symmetry breaking} \label{sec:hessian}
In this section we prove Theorem \ref{2speciessymmetrybreaking}, which generalizes the de Almeida--Thouless condition for symmetry breaking in the SK model \cite{almeida-thouless78}.
The proof is a perturbative argument following the strategy of \cite{toninelli02}, in which a symmetry breaking parameter is introduced alongside the RS critical point---whose uniqueness was established in Section \ref{sec:uniqueness}---by bringing the latter's atomic weight $\zeta$ just slightly away from $1$.
A nice exposition of the original argument in the single-species case can by found in \cite[Section 13.3]{talagrand11II}.

Assume $h > 0$ and fix an RS critical point $q_*\in\c(\beta,h)$.
For any $p\in[0,1]^M$ such that $p^s\geq q_*^s$ for each $s$, we can define
\eq{
V(p) = \frac{\partial \mathscr{P}_\mathrm{1RSB}(q_*,p,\zeta)}{\partial \zeta}\Big|_{\zeta = 1}.
}
Here $\mathscr{P}_\mathrm{1RSB}$ is the Parisi functional restricted to level-$1$ symmetry breaking.
An explicit expression \eqref{parisi_1level} is computed in Appendix \ref{sec:appendix}.
The quantity $V(p)$ is useful because of the following observations.

\begin{lemma} \label{V_calculations}
Fix any $q = q_*\in\c(\beta,h)$, and recall $Q_1^s$ defined by \eqref{Qdef}.
With the notation
\eq{
\gamma_s = \lambda_s\EE\sech^4(\beta\eta\sqrt{Q_1^s}+h), \quad s = 1,\dots,M,
}
the following equalities hold:
\begin{itemize}
\item[(a)] $V(q_*) = 0$
\item[(b)] $\nabla V(q_*) = 0$
\item[(c)] $HV(q_*) = \beta^{2}\Lambda(2\beta^{2}\Delta^{2}\Gamma\Delta^{2}-\Delta^{2})\Lambda$, where $\Lambda$ and $\Gamma$ are the $M\times M$ diagonal matrices with diagonal entries $(\lambda_s)_{s=1}^M$ and $(\gamma_s)_{s=1}^M$, respectively.
\end{itemize}
\end{lemma}
The calculations in Lemma \ref{V_calculations} are straightforward generalizations of the (elegant but somewhat tricky) procedure found in \cite[Section 13.3]{talagrand11II}, and thus postponed to Appendix \ref{sec:appendix}.
Most important, part (c) identifies a condition for symmetry breaking once we note the following result.

\begin{cor} \label{general_RSB_condition}
Assume \eqref{variational} and that $\mathrm{RS}(\beta,h) = \mathscr{P}_\mathrm{RS}(q_*)$ for some $q_*\in\c(\beta,h)$.
If there exists a $x\in\R^M$ with all nonnegative entries such that $x^\intercal HV(q_*)x>0$, then
\eeq{
\lim_{N\to\infty} F_N < \mathrm{RS}(\beta,h). \label{RSB_conclusion}
}
\end{cor}

\begin{proof}
First note that by \eqref{RS_condition}, we must have $q_*\in[0,1)^M$.
Since $\nabla V(q_*) = 0$ and $x^\intercal HV(q_*)x > 0$, there exists $\eps > 0$ small enough that $V(q_*+\eps x) > V(q_*) = 0$.
That is,
\eq{
\frac{\partial \mathscr{P}_\mathrm{1RSB}(q_*,q_*+\eps x,\zeta)}{\partial \zeta}\Big|_{\zeta=1} > 0,
}
implying there is $\zeta < 1$ such that
\eq{
\mathscr{P}_\mathrm{1RSB}(q_*,q_*+\eps x,\zeta) < \mathscr{P}_\mathrm{1RSB}(q_*,q_*+\eps x,1) \stackrel{\footnotesize{\mbox{\eqref{RS_equivalence}}}}{=} \mathscr{P}_\mathrm{RS}(q_*)
= \mathrm{RS}(\beta,h).
}
Because of \eqref{variational}, \eqref{RSB_conclusion} follows.
\end{proof}

It is apparent from Corollary \ref{general_RSB_condition} that in order to obtain the correct AT condition for multi-species models, one must have good understanding of the set $\c(\beta,h)$.
Thanks to Theorem \ref{uniquenessof2speciesrRSsol}, this has been accomplished in the two-species case for $h>0$, allowing us to proceed with the following argument.

\begin{proof}[Proof of Theorem \ref{2speciessymmetrybreaking}]
We return to the case $M=2$.
We know from Theorem \ref{uniquenessof2speciesrRSsol} that $q_* \in \c(\beta,h)$ is unique, and moreover that the hypothesis of Corollary \ref{general_RSB_condition} holds.
Therefore, it suffices to show that \eqref{1RSB_condition} implies the existence of $x\in\R^2$ with nonnegative entries such that $x^\intercal H x > 0$, where $H = HV(q_*)$.
To simplify the task, we let $K \coloneqq 2\beta^{2}\Delta^{2}\Gamma\Delta^{2}-\Delta^{2}$ and note that
\eq{
x^\intercal K x > 0 \quad \implies \quad (\Lambda^{-1}x)^\intercal H (\Lambda^{-1}x)
= \beta^2 x^\intercal K x > 0.
}
Therefore, we can replace $H$ by $K$, since multiplication by $\Lambda^{-1}$ preserves nonnegativity of coordinates.

More specifically, we have $K=\begin{pmatrix}u &v \\ v &t\end{pmatrix}$, where
\eq{
u&=2\beta^{2}(\gamma_{1}(\Delta_{11}^{2})^2+\gamma_{2})-\Delta_{11}^2\\
t&=2\beta^{2}(\gamma_{1}+\gamma_{2}(\Delta_{22}^{2})^{2})-\Delta_{22}^{2} \\
v&=2\beta^{2}(\gamma_{1}\Delta_{11}^2+\gamma_{2}\Delta_{22}^2)-1.
}
Now, $x^\intercal K x > 0$ for some $x$ with nonnegative entries if and only if at least one of the following three inequalities is true:
\begin{subequations}
\label{2speciescondition}
\begin{align}
u&>0\quad\text{or} \label{2speciescondition_1}\\
t&>0\quad\text{or} \label{2speciescondition_2}\\
u,t &\leq 0 \quad\text{and}\quad \sqrt{ut} <v \label{2speciescondition_3}. 
\end{align}
\end{subequations}
Direct computation shows that \eqref{2speciescondition} is equivalent to
\begin{subequations} \label{2speciescondition_new}
\begin{align}
2\beta^2 &> 2\beta^{2}_{u}\coloneqq\frac{\Delta_{11}^2}{\gamma_{1}(\Delta_{11}^2)^{2}+\gamma_{2}} \quad \text{or} \\
2\beta^2 &>2\beta^{2}_{t}\coloneqq\frac{\Delta_{22}^2}{\gamma_{1}+\gamma_{2}(\Delta_{22}^2)^{2}} \quad \text{or} \\
\min(2\beta^2_u,2\beta^2_t)\geq2\beta^2&>2\beta^{2}_{v}\coloneqq\frac{1}{\gamma_{1}\Delta_{11}^2+\gamma_{2}\Delta_{22}^2} 
\quad \text{and} \quad 2\beta^{2}_{m}<2\beta^{2} <2\beta^{2}_{M},
\end{align}
\end{subequations}
where
\eq{
2\beta^{2}_{m}&=\frac{2}{\gamma_1 \Delta_{11}^2 +\gamma_2 \Delta_{22}^2 +\sqrt{(\gamma_1 \Delta_{11}^2 -\gamma_2 \Delta_{22}^2)^2 +4 \gamma_1 \gamma_2}} \\
2\beta^{2}_{M}&=\frac{2}{\gamma_1 \Delta_{11}^2 +\gamma_2 \Delta_{22}^2 -\sqrt{(\gamma_1 \Delta_{11}^2 -\gamma_2 \Delta_{22}^2)^2 +4 \gamma_1 \gamma_2}}.
}
Note that
\eq{
(\gamma_1\Delta_{11}^2-\gamma_2\Delta_{22}^2)^2+4\gamma_1\gamma_2
&< (\gamma_1\Delta_{11}^2-\gamma_2\Delta_{22}^2)^2+4\gamma_1\gamma_2\Delta_{11}^2\Delta_{22}^2 \\
&= (\gamma_1\Delta_{11}^2+\gamma_2\Delta_{22}^2)^2,
}
which ensures $0<\beta^2_v < \beta^2_m < \beta^2_M$.
We also claim that
\eeq{ \label{order_claim}
\beta^2_m < \min(\beta_u^2,\beta_t^2) \leq \max(\beta_u^2,\beta_t^2) < \beta^2_M.
}
For instance, suppose $\gamma_1\Delta_{11}^2 \leq \gamma_2\Delta_{22}^2$.
Then
\eq{
\gamma_1\Delta_{11}^2(\Delta_{11}^2\Delta_{22}^2-1) &\leq \gamma_2\Delta_{22}^2(\Delta_{11}^2\Delta_{22}^2-1) \\
\implies \quad (\gamma_1(\Delta_{11}^2)^2+\gamma_2)\Delta_{22}^2 &\leq (\gamma_1+\gamma_2(\Delta_{22}^2)^2)\Delta_{11}^2 \\
\implies \quad
2\beta_t^2 = \frac{\Delta_{22}^2}{\gamma_1 + \gamma_2(\Delta_{22}^2)^2} &\leq \frac{\Delta_{11}^2}{\gamma_1(\Delta_{11}^2)^2+\gamma_2} = 2\beta_u^2.
}
To establish the first bound in \eqref{order_claim}, we observe that $\beta_t^2 > \beta_m^2$ if and only if
\eq{
\Delta_{22}^2\Big(\gamma_1\Delta_{11}^2+\gamma_2\Delta_{22}^2+\sqrt{(\gamma_1\Delta_{11}^2-\gamma_2\Delta_{22}^2)^2+4\gamma_1\gamma_2}\, \Big)&> 2\big(\gamma_1+\gamma_2(\Delta_{22}^2)^2\big) \\
\iff \quad \Delta_{22}^2\sqrt{(\gamma_1\Delta_{11}^2-\gamma_2\Delta_{22}^2)^2+4\gamma_1\gamma_2} &> \gamma_1(2-\Delta_{11}^2\Delta_{22}^2)+\gamma_2(\Delta_{22}^2)^2 \\
\iff \quad 0 &> 4\gamma_1^2(1-\Delta_{11}^2\Delta_{22}^2),
}
which is true by \eqref{delta_assumptions}.
To establish the last bound in \eqref{order_claim}, we drop the term $4\gamma_1\gamma_2$ from the denominator of $\beta_M^2$:
\eq{
2\beta_M^2 > \frac{2}{\gamma_1\Delta_{11}^2+\gamma_2\Delta_{22}^2-\sqrt{(\gamma_1\Delta_{11}^2-\gamma_2\Delta_{22}^2)^2}}
= \frac{1}{\gamma_1\Delta_{11}^2} > 2\beta_u^2.
}
We have thus proved \eqref{order_claim} under the assumption $\gamma_1\Delta_{11}^2\leq\gamma_2\Delta_{22}^2$, but the proof is analogous in the reverse case.
Finally, because of \eqref{order_claim}, we see that \eqref{2speciescondition_new} is equivalent to the single condition $\beta^{2}>\beta^{2}_{m}$.
\end{proof}


%
%

\section{Acknowledgments}
We are grateful to Amir Dembo and Andrea Montanari for their advice and encouragement on this project.
We thank Antonio Auffinger, Erwin Bolthausen, and Aukosh Jagannath for their insights and feedback, and the referee for several suggestions to improve the manuscript.

\bibliography{spin_glasses}

\appendix

\section{Proof of Lemma \ref{V_calculations}} \label{sec:appendix}
\begin{proof}[Proof of Lemma \ref{V_calculations}]
Here we consider \eqref{parisi_expression} when $k = 1$, and
\eq{
\zeta_1 = \zeta \in (0,1), \quad
q_1^s = q^s, \quad 
q_2^s = p^s.
}
Recall that with these choices, we have $Q_0 = Q_0^s = 0$ and
\eq{
Q_1^s &= 2\sum_t \Delta_{st}^2\lambda_tq^t &
Q_1 &= \sum_{s,t}\Delta_{st}^2\lambda_s\lambda_tq^sq^t \\
Q_2^s &= 2\sum_{t}\Delta_{st}^2\lambda_tp^t &
Q_2 &= \sum_{s,t}\Delta_{st}^2\lambda_s\lambda_tp^sp^t \\
Q_3^s &= 2\sum_{t}\Delta_{st}^2\lambda_t &
Q_3 &= \sum_{s,t}\Delta_{st}^2\lambda_s\lambda_t.
}
We have
\eq{
\mathscr{P}_\mathrm{1RSB}(q,p,\zeta) &\coloneqq \log 2 + \sum_s \lambda_s X_0^s - \frac{\beta^2}{2}\sum_{\ell=1}^2 \zeta_\ell(Q_{\ell+1}-Q_{\ell}) \\
&= \log 2 + \sum_s \lambda_s X_0^s - \frac{\beta^2}{2}(Q_3 - Q_2+\zeta(Q_2-Q_1)),
}
where
\eq{
X_3^s &= 
\log \cosh(\beta\eta_3\sqrt{Q_3^s-Q_2^s} + \beta\eta_2\sqrt{Q_2^s-Q_1^s}+\beta\eta_1\sqrt{Q_1^s}+h) \\
\implies X_2^s &= \log \EE_3 \cosh(\beta\eta_3\sqrt{Q_3^s-Q_2^s} + \beta\eta_2\sqrt{Q_2^s-Q_1^s}+\beta\eta_1\sqrt{Q_1^s}+h) \\
&= \frac{\beta^2}{2}(Q_3^s-Q_2^s) + \log\cosh(\beta\eta_2\sqrt{Q_2^s-Q_1^s}+\beta\eta_1\sqrt{Q_1^s}+h) \\
\implies X_1^s &= \frac{1}{\zeta}\log\EE_2 \exp(\zeta_2 X_2^s) \\
&= \frac{\beta^2}{2}(Q_3^s-Q_2^s) + \frac{1}{\zeta}\log \EE_2 \cosh^\zeta(\beta\eta_2\sqrt{Q_2^s-Q_1^s} + \beta\eta_1\sqrt{Q_1^s}+h) \\
\implies X_0^s &= \EE X_1^s = \frac{\beta^2}{2}(Q_3^s-Q_2^s) + \frac{1}{\zeta} \EE_1\log\EE_2 \cosh^\zeta(\beta\eta_2\sqrt{Q_2^s-Q_1^s} + \beta\eta_1\sqrt{Q_1^s}+h).
}
In simplifying $X_2^s$, we have used the fact that for $\eta \sim \n(0,1)$ and $\sigma>0$,
\eeq{ \label{gaussian_calculation}
\EE \cosh(\sigma \eta +h) &= \frac{1}{2}\EE(\e^{\sigma \eta+h} + \e^{-\sigma \eta-h}) \\
&= \frac{1}{2}(\e^{\sigma^2/2 + h} + \e^{\sigma^2/2 - h}) = \e^{\sigma^2/2}\cosh(h).
}
In summary,
\eeq{ \label{parisi_1level}
\mathscr{P}_\mathrm{1RSB}(q,p,\zeta) &= \log 2+ \sum_s \lambda_s \frac{1}{\zeta} \EE_1\log\EE_2 \cosh^\zeta(\beta\eta_2\sqrt{Q_2^s-Q_1^s} + \beta\eta_1\sqrt{Q_1^s}+h) \\ 
&\phantom{=} +\frac{\beta^2}{2}\sum_s \lambda_s (Q_3^s - Q_2^s) - \frac{\beta^2}{2}(Q_3-Q_2+\zeta(Q_2-Q_1)).
}
Notice that when $\zeta = 1$, we recover the replica symmetric expression \eqref{RS_formula}:
\eeq{ \label{RS_equivalence}
\mathscr{P}_\mathrm{1RSB}(q,p,1) &= \log 2 + \sum_s\lambda_s \EE_1\log\EE_2\cosh(\beta\eta_2\sqrt{Q_2^s-Q_1^s} + \beta\eta_1\sqrt{Q_1^s}+h) \\
&\phantom{=}+ \frac{\beta^2}{2} \sum_s \lambda_s(Q_3^s - Q_2^s)- \frac{\beta^2}{2}(Q_3-Q_2+Q_2-Q_1) \\
&= \log 2 + \sum_s \lambda_s \EE_1\bigg[\frac{\beta^2}{2}(Q_2^s-Q_1^s) + \log \cosh(\beta\eta_1\sqrt{Q_1^s}+h)\bigg] \\
&\phantom{=}+ \frac{\beta^2}{2} \sum_s \lambda_s(Q_3^s - Q_2^s) - \frac{\beta^2}{2}(Q_3-Q_1) \\
&= \log 2 + \sum_s \lambda_s \EE_1\log \cosh(\beta\eta_1\sqrt{Q_1^s}+h) \\
&\phantom{=}+ \frac{\beta^2}{2}\sum_s \lambda_s(Q_3^s-Q_1^s) - \frac{\beta^2}{2}(Q_3-Q_1) = \mathscr{P}_\mathrm{RS}(q).
}
Henceforth fix an RS critical point $q = q_*\in\c(\beta,h)$.
For ease of notation, let us write 
\eq{
Y_1^s &\coloneqq \beta\eta_1\sqrt{Q_1^s}+h, \qquad Y_2^s \coloneqq \beta\eta_2\sqrt{Q_2^s-Q_1^s} +Y_1^s,
}
so that
\eq{
\mathscr{P}_\mathrm{1RSB}(q,p,\zeta) &= \log 2+ \sum_s \lambda_s \frac{1}{\zeta} \EE_1\log\EE_2 \cosh^\zeta Y_2^s +\frac{\beta^2}{2}\sum_s \lambda_s (Q_3^s - Q_2^s) \\
&\phantom{=}- \frac{\beta^2}{2}(Q_3-Q_2+\zeta(Q_2-Q_1)).
}
We can then calculate
\eq{
\frac{\partial \mathscr{P}_\mathrm{1RSB}(q_*,p,\zeta)}{\partial \zeta}
&= \sum_s \lambda_s\bigg[\frac{-\EE_1\log\EE_2 \cosh^\zeta Y_2^s}{\zeta^2} +\EE_1\bigg(\frac{\EE_2\log(\cosh Y_2^s)\cosh^\zeta Y_2^s}{\zeta\EE_2 \cosh^\zeta Y_2^s}\bigg)\bigg] \\
&\phantom{=}- \frac{\beta^2}{2}(Q_2-Q_1),
}
which gives
\eeq{ \label{pre_deriv}
V(p) &= \sum_s\lambda_s\bigg[-\EE_1\log\EE_2 \cosh Y_2^s + \EE_1\bigg(\frac{\EE_2\log(\cosh Y_2^s)\cosh Y_2^s}{\EE_2\cosh Y_2^s}\bigg)\bigg] \\
&\phantom{=} - \frac{\beta^2}{2}(Q_2-Q_1).
}
When $p = q_*$, we have $Y_2^s = Y_1^s$ and $Q_2 = Q_1$.
In particular, $Y_2^s$ has no dependence on $\eta_2$, and so the above expression reduces to $V(q_*) = 0$.
This proves claim (a).

The next step is to take partial derivatives with respect to the $p^t$.
First, we use \eqref{gaussian_calculation} to make the simple calculation
\eeq{ 
\EE_2 \cosh Y_2^s = \exp\Big(\frac{\beta^2}{2}(Q_2^s-Q_1^s)\Big)\cosh Y_1^s. \label{Y_switch}
}
Since $Y_1^s$ has no dependence on $p$, and
\eeq{ \label{Q_deriv}
\frac{\partial}{\partial p^t}(Q_2^s-Q_1^s) = \frac{\partial}{\partial p^t}Q_2^s = 2\Delta_{st}^2\lambda_t,
}
we find
\eeq{ \label{deriv_1}
\frac{\partial}{\partial p^t} \EE_1\log\EE_2\cosh Y_2^s 
=\frac{\partial}{\partial p^t}\Big[\frac{\beta^2}{2}(Q_2^s-Q_1^s)+\EE_1\log\cosh Y_1^s\Big]
= \beta^2\Delta_{st}^2\lambda_t.
}
Next we observe that for any twice differentiable function whose derivatives have at most exponential growth at infinity, \eqref{Q_deriv} and Gaussian integration by parts together give
\eeq{ \label{f_deriv}
\frac{\partial}{\partial p^t} \EE_2 f(Y_2^s) 
=  \beta\Delta_{st}^2\lambda_t\EE_2\bigg(f'(Y_2^s)\frac{\eta_2}{\sqrt{Q_2^s-Q_1^s}}\bigg)
&= \beta^2\Delta_{st}^2\lambda_t\EE_2 f''(Y_2^s).
}
Hence
\eeq{ \label{expectation_formula}
&\frac{\partial}{\partial p^t} \EE_1\bigg(\frac{\EE_2f(Y_2^s)}{\EE_2\cosh Y_2^s}\bigg)
\stackrel{\mbox{\footnotesize\eqref{Y_switch}}}{=} \frac{\partial}{\partial p^t}\bigg[ \exp\Big(-\frac{\beta^2}{2}(Q_2^s-Q_1^s)\Big)\EE_1\bigg(\frac{\EE_2 f(Y_2^s)}{\cosh Y_1^s}\bigg)\bigg] \\
&\stackrel{\mbox{\footnotesize\eqref{Q_deriv},\eqref{f_deriv}}}{=} \beta^2\Delta_{st}^2\lambda_t\exp\Big(-\frac{\beta^2}{2}(Q_2^s-Q_1^s)\Big)\bigg[- \EE_1\bigg(\frac{\EE_2 f(Y_2^s)}{\cosh Y_1^s}\bigg) + \EE_1\bigg(\frac{\EE_2f''(Y_2^s)}{\cosh Y_1^s}\bigg)\bigg] \\
&\stackrel{\phantom{\mbox{\footnotesize\eqref{Q_deriv},\eqref{f_deriv}}}}{=} \beta^2\Delta_{st}^2\lambda_t\exp\Big(-\frac{\beta^2}{2}(Q_2^s-Q_1^s)\Big)\EE_1\bigg(\frac{\EE_2[f''(Y_2^s)-f(Y_2^s)]}{\cosh Y_1^s}\bigg) \\
&\stackrel{\hspace{2.5ex}\mbox{\footnotesize\eqref{Y_switch}}\hspace{2.5ex}}{=} \beta^2\Delta_{st}^2\lambda_t\EE_1\bigg(\frac{\EE_2[f''(Y_2^s)-f(Y_2^s)]}{\EE_2\cosh Y_2^s}\bigg).
}
We now apply \eqref{expectation_formula} to $f(x) = \log(\cosh x)\cosh x$, for which
\eq{
f''(x)-f(x) &= \cosh x + \sinh x \tanh x,
}
to obtain
\eeq{ \label{deriv_2}
\frac{\partial}{\partial p^t}\EE_1\bigg(\frac{\EE_2 \log(\cosh Y_2^s) \cosh Y_2^s}{\EE_2\cosh Y_2^s}\bigg)\bigg] 
= \beta^2\Delta_{st}^2\lambda_t\bigg[1 + \EE_1\bigg(\frac{\EE_2 \sinh Y_2^s \tanh Y_2^s}{\EE_2 \cosh Y_2^s}\bigg)\bigg].
}
Finally, we have
\eeq{ \label{deriv_3}
\frac{\partial}{\partial p^t}(Q_2 - Q_1) = \frac{\partial}{\partial p^t}Q_2 = 2\sum_s\Delta_{st}^2\lambda_s\lambda_t p^s. 
}
Using \eqref{deriv_1}, \eqref{deriv_2}, and \eqref{deriv_3} in \eqref{pre_deriv}, we arrive at
\eq{
\frac{\partial}{\partial p^t}V(p) = \beta^2\lambda_t\sum_s \Delta_{st}^2\lambda_s\bigg[\EE_1\bigg(\frac{\EE_2 \sinh Y_2^s \tanh Y_2^s}{\EE_2 \cosh Y_2^s}\bigg) - p^s\bigg].
}
Once more, if $p = q_*$, then $Y_2^s = Y_1^s$ has no dependence on $\eta_2$, in which case
\eq{
\EE_1\bigg(\frac{\EE_2 \sinh Y_2^s \tanh Y_2^s}{\cosh Y_1^s}\bigg) - p^s
= \EE_1(\tanh^2 Y_1^s) - q_*^s = q_*^s - q_*^s = 0 \quad \text{for all $s$.}
}
Consequently, claim (b) holds: $\nabla V(q_* ) = 0$.

Our final step is to compute the Hessian of $V$.
We have
\eq{
\frac{\partial^2}{\partial p^{t'}\partial p^t}V(p) = \beta^2\lambda_t\sum_{s}\Delta^2_{st}\lambda_s\bigg[\frac{\partial}{\partial p^{t'}} \EE_1\bigg(\frac{\EE_2\sinh Y_2^s \tanh Y_2^s}{\EE_2\cosh Y_2^s}\bigg) - \delta_{st'}\bigg],
}
where $\delta_{st'} = 1$ if $s = t'$ and $0$ zero otherwise.
To determine the derivative of the expectation, we apply \eqref{expectation_formula} with $f(x) = \sinh x \tanh x$, for which 
\eq{
f''(x)-f(x) &= 2\mathrm{sech}^3\,x.
}
After doing so, we arrive at
\eq{
\frac{\partial^2}{\partial p^{t'}\partial p^{t}}V(p) &= \beta^2\lambda_t\sum_s \Delta^2_{st}\lambda_s\bigg[2\beta^2\Delta_{st'}^2\lambda_{t'}\EE_1\bigg(\frac{\EE_2\sech^3 Y_2^s}{\EE_2\cosh Y_2^s}\bigg) - \delta_{st'}\bigg].
\intertext{As before, the expression simplifies when $p=q^*$, since then $Y_2^s=Y_1^s$ has no dependence on $\eta_2$.
Namely,}
\frac{\partial^2}{\partial p^{t'}\partial p^{t}}V(q_*)
&= \beta^2\lambda_t\sum_{s}\Delta_{st}^2\lambda_s\big[2\beta^2\Delta_{st'}^2\lambda_{t'}\EE_1\, \mathrm{sech}^4\, Y_1^s - \delta_{st'}\big] \\
&= \bigg(2\beta^4\lambda_t\lambda_{t'}\sum_s \Delta_{st}^2\Delta_{st'}^2\lambda_s \EE_1\, \mathrm{sech}^4\, Y_1^s\bigg) - \beta^2\lambda_t\lambda_{t'}\Delta_{tt'}^2.
}
Rewriting the expression in terms of matrices yields claim (c).
\end{proof}

\end{document}